\documentclass[11pt]{amsart}

\usepackage{amssymb}
\usepackage{mathrsfs}
\usepackage{amsfonts}
\usepackage{latexsym,amsmath,amsthm,amssymb,amsfonts}
\usepackage[usenames]{color}
\usepackage{xspace,colortbl}
\usepackage{graphicx}
\usepackage{tipa}

\newcommand{\be}{\begin{equation}}
\newcommand{\ee}{\end{equation}}
\newcommand{\beq}{\begin{eqnarray}}
\newcommand{\eeq}{\end{eqnarray}}

\newtheorem{prop}{Proposition}[section]

\newtheorem{remark}[prop]{Remark}

\def\begeq{\begin{equation}}
\def\endeq{\end{equation}}

\def\odot{\setbox0=\hbox{$\bigcirc$}\relax \mathbin {\hbox
to0pt{\raise.5pt\hbox to\wd0{\hfil $\wedge$\hfil}\hss}\box0 }}

\numberwithin{equation} {section}

\numberwithin{equation}{section}
\newtheorem{theorem}{\bf Theorem}[section]

\newtheorem{lemma}[theorem]{\bf Lemma}

\newtheorem{corollary}[theorem]{\bf Corollary}

\allowdisplaybreaks

\begin{document}

\title[star-shaped hypersurfaces evolving in a cone]
 {A class of inverse curvature flows for star-shaped hypersurfaces evolving in a cone}

\author{
  Jing Mao$^{1,2,\ast}$, Qiang Tu$^{1,\ast}$}
\address{
$^{1}$ Faculty of Mathematics and Statistics, Key Laboratory of
Applied Mathematics of Hubei Province, Hubei University, Wuhan
430062, China. }

\address{$^{2}$  Department of Mathematics, Instituto Superior T\'{e}cnico, University of Lisbon,
Av. Rovisco Pais, 1049-001 Lisbon, Portugal}

\email{jiner120@163.com, qiangtu@whu.edu.cn}

\thanks{$\ast$ Corresponding authors}

\date{}
\begin{abstract}
Given a smooth convex cone in the Euclidean $(n+1)$-space
($n\geq2$), we consider strictly mean convex hypersurfaces with
boundary which are star-shaped with respect to the center of the
cone and which meet the cone perpendicularly. If those hypersurfaces
inside the cone evolve by a class of inverse curvature flows, then,
by using the convexity of the cone in the derivation of the
gradient and H\"{o}lder estimates, we can prove
that this evolution exists for all the time and the evolving
hypersurfaces converge smoothly to a piece of a round sphere as time
tends to infinity.
\end{abstract}

\maketitle {\it \small{{\bf Keywords}: Inverse curvature flows,
star-shaped, Neumann boundary value problem.}

}

\section{Introduction}

Recently, Chen, Mao, Tu and Wu \cite{cmtw} considered the evolution
of a one-parameter family of closed, star-shaped and strictly mean
convex hypersurfaces $M_{t}^{n}$, given by $X(\cdot, t):
\mathbb{S}^{n}\times[0,T)\rightarrow \mathbb{R}^{n+1}$ with some
$T<\infty$, under the flow
\begin{equation}\label{Eq-1}
\left\{
\begin{aligned}
&\frac{\partial }{\partial t}X=\frac{1}{|X|^{\alpha}H(X)}\nu,\\\
&X(\cdot,0)=M_{0}^{n},
\end{aligned}
\right.
\end{equation}
where $\nu$ is the unit outward normal vector of $M_{t}^{n}$, $H$ is
the mean curvature of $M_{t}^{n}$, and $|X|$ is the distance from
the point $X(x, t)$ to the origin of $\mathbb{R}^{n+1}$. For
$\alpha\geq0$, they showed the long-time existence and the
asymptotical behavior of the flow \eqref{Eq-1}. Clearly, when
$\alpha=0$, the flow \eqref{Eq-1} degenerates into the classical
inverse mean curvature flow (IMCF for short), and therefore
Gerhardt's or Urbas's classical result for the IMCF in
$\mathbb{R}^{n+1}$ (see \cite{Ge90,Ur}) is covered by the main
conclusion of \cite{cmtw} as a special case.

As we know, the classical IMCF is scale invariant. However,
generally, the flow (\ref{Eq-1}) is non-scale-invariant except the
case $\alpha=0$. The meaning of studying the non-scale-invariant
version of inverse curvature flows (ICFs for short) have been
revealed clearly by Gerhardt \cite{Ge14}, where he investigated the
non-scale-invariant version of the classical ICF considered by
himself in \cite{Ge90}. This improvement from the scale invariant
case to the non-scale-invariant case permits that Gerhardt's main
conclusion in \cite{Ge14} covers some interesting conclusions in
\cite{qrl,ocs} for the inverse Gauss curvature flow (IGCF for short)
and the power of the IGCF. Based on this reason, it also should be
interesting to investigate properties of the non-scale-invariant
flow (\ref{Eq-1}) in different settings -- for this purpose, please
see the series work \cite{cmtw,chmw} of Prof. Jing Mao and his
collaborators. The flow (\ref{Eq-1}) is an initial value problem of
second-order parabolic PDEs. \emph{Could we consider the case of
boundary value problems?} This motivation forces us to consider the
evolution of hypersurfaces with boundary under the ICFs considered
in \cite{cmtw}.

Marquardt \cite{Mar} considered the classical IMCF with a Neumann
boundary condition (NBC for short), where the embedded flowing
hypersurfaces were supposed to be perpendicular to a smooth convex
cone in $\mathbb{R}^{n+1}$. He proved that the flow exists for all
the time and after rescaling, the evolving
 hypersurfaces converge smoothly  to a piece of a round sphere.
Later, Lamber and Scheuer \cite{La} extended this interesting
conclusion to the situation that the hypersurfaces are perpendicular
to the prescribed sphere. In 2017, Chen, Mao, Xiang and Xu
\cite{cmxx} improved Marquardt's main conclusion above to the case
that the ambient space is the warped product
$I\times_{\lambda(r)}N^{n}$, where $I\subseteq\mathbb{R}$ is an
unbounded interval of $\mathbb{R}$, $N^{n}$ is an $n$-dimensional
Riemannian manifold with nonnegative Ricci curvature, and the
warping function $\lambda(r)$ satisfies some growth assumptions.
Inspired by these works,
 it should be interesting to consider the flow (\ref{Eq-1}) with a
 prescribed NBC. In fact, we can prove the following:

\begin{theorem}\label{main1.1}
Let $\alpha>0$, $M^n\subset\mathbb{S}^n$ be some piece of the unit
sphere $\mathbb{S}^{n}\subset\mathbb{R}^{n+1}$, and
$\Sigma^n:=\{rx\in \mathbb{R}^{n+1}| r>0, x\in \partial M^n\}$ be
the boundary of a smooth, convex cone that is centered at the origin
and has outward unit normal $\mu$. Let
$X_{0}:M^{n}\rightarrow\mathbb{R}^{n+1}$ such that
$M_{0}^{n}:=X_{0}(M^{n})$ is a compact, strictly mean convex
$C^{2,\gamma}$-hypersurface ($0<\gamma<1$) which is star-shaped with
respect to the center of the cone.
 Assume that
 \begin{eqnarray*}
M_{0}^{n}=\mathrm{graph}_{M^n}u_{0}
 \end{eqnarray*}
 is a graph over $M^n$ for a positive
map $u_0: M^n\rightarrow \mathbb{R}$ and
 \begin{eqnarray*}
\partial M_{0}^{n}\subset \Sigma^n, \qquad
\langle\mu\circ X_{0}, \nu_0\circ X_{0} \rangle|_{\partial M^n}=0,
 \end{eqnarray*}
 where $\nu_0$ is the unit normal to $M_{0}^{n}$. Then

(i) there exists a family of strictly mean convex hypersurfaces
$M_{t}^{n}$ given by the unique embedding
\begin{eqnarray*}
X\in C^{2+\gamma,1+\frac{\gamma}{2}} (M^n\times [0,\infty),
\mathbb{R}^{n+1}) \cap C^{\infty} (M^n\times (0,\infty),
\mathbb{R}^{n+1})
\end{eqnarray*}
with $X(\partial M^n, t) \subset \Sigma^n$ for $t\geq 0$, satisfying the following system
\begin{equation}\label{Eq}
\left\{
\begin{aligned}
&\frac{\partial }{\partial t}X=\frac{1}{|X|^{\alpha}H}\nu
&&in~
M^n \times(0,\infty)\\
&\langle \mu \circ X, \nu\circ X\rangle =0&&on~ \partial M^n \times(0,\infty)\\
&X(\cdot,0)=M_{0}^{n}  && in~M^n
\end{aligned}
\right.
\end{equation}
where $H$ is the mean curvature of
$M_{t}^{n}:=X(M^n,t)=X_{t}(M^{n})$, $\nu$ is the unit outward normal
vector of $M_{t}^{n}$ and $|X|$ is the distance from the point $X(x,
t)$ to the origin. Moreover, the H\"{o}lder norm on the parabolic
space $M^n\times(0,\infty)$ is defined in the usual way (see, e.g.,
\cite[Note 2.5.4]{Ge3}).

(ii) the leaves $M_{t}^{n}$ are graphs over $M^n$, i.e.,
 \begin{eqnarray*}
M_{t}^{n}=\mathrm{graph}_{M^n}u(\cdot, t).
\end{eqnarray*}

(iii) Moreover, the evolving hypersurfaces converge smoothly after
rescaling to a piece of a round sphere of radius $r_{\infty}$, where
$r_{\infty}$ satisfies
\begin{eqnarray*}
\frac{1}{\sup\limits_{M^{n}}u_{0}}\left(\frac{\mathcal{H}^n(M_{0}^{n})}{\mathcal{H}^n(M^{n})}\right)^{\frac{1}{n}}\leq
r_{\infty}
\leq\frac{1}{\inf\limits_{M^{n}}u_{0}}\left(\frac{\mathcal{H}^n(M_{0}^{n})}{\mathcal{H}^n(M^{n})}\right)^{\frac{1}{n}},
\end{eqnarray*}
where $\mathcal{H}^n(\cdot)$ stands for the $n$-dimensional
Hausdorff measure of a prescribed $n$-manifold.
\end{theorem}

\begin{remark}
\rm{ In order to avoid any potential confusion with the mean
curvature
 $H$, we use $C^{m+2+\gamma,\frac{m+2+\gamma}{2}}$ not
$H^{m+2+\gamma,\frac{m+2+\gamma}{2}}$ used in \cite{Ge14} to
represent the parabolic H\"{o}lder norm. It is easy to check that
all the arguments in the sequel are still valid for the case
$\alpha=0$ except some minor changes should be made. For instance,
if $\alpha=0$, then (\ref{blow}) becomes $\phi(x,t)=\frac{1}{n}t+c$.
However, in this setting, one can also get the $C^0$ estimate as
well. Clearly, when $\alpha=0$, the flow (\ref{Eq}) degenerates into
the parabolic system with the vanishing NBC in \cite[Theorem
1]{Mar}, and correspondingly, our main conclusion here covers
\cite[Theorem 1]{Mar} as a special case. }
\end{remark}

\section{The corresponding scalar equation}

\subsection{The geometry of graphic hypersurfaces}

\

For an $n$-dimensional Riemannian manifold $(M^n,g)$, the Riemann
curvature (3,1)-tensor $Rm$ is defined by
\begin{equation*}
Rm(X, Y)Z=-\nabla_{X}\nabla_{Y}Z+\nabla_{Y}\nabla_{X}Z+\nabla_{[X,
Y]}Z.
\end{equation*}
Pick a local coordinate chart $\{x^{i}\}_{i=1}^{n}$ of $M$. The
component of the (3,1)-tensor $Rm$ is defined by
\begin{equation*}
Rm\bigg({\frac{\partial}{\partial x^i}}, {\frac{\partial}{\partial x^j}}\bigg){\frac{\partial}{\partial x^k}}
\doteq R_{ijk}^{l}{\frac{\partial}{\partial x^l}}
\end{equation*}
and $R_{ijkl}:= g_{lm}R_{ijk}^{m}$. Then, we have the standard
commutation formulas (the Ricci identities):
\begin{equation*}
(\nabla_{i}\nabla_{j}-\nabla_{j}\nabla_{i})\alpha_{k_{1}\cdot\cdot\cdot
k_{r}} =\sum_{l=1}^{r}R^{m}_{ijk_{l}} \alpha_{k_{1}\cdot\cdot\cdot
k_{l-1}m k_{l+1}\cdot\cdot\cdot k_{r}}.
\end{equation*}
If furthermore $(M^n, g)$ is an immersed hypersurface in
$\mathbb{R}^{n+1}$ with $R_{ijkl}$ the Riemannian curvature of $M$.
Let $\nu$ be a given unit outward normal vector and $h_{ij}$ be the
second fundamental form of the hypersurface $M$ with respect to
$\nu$, that is,
 \begin{eqnarray*}
h_{ij}=-\left\langle\frac{\partial^2 X}{\partial x^i\partial x^j},
\nu\right\rangle_{\mathbb{R}^{n+1}}.
 \end{eqnarray*}
 Set $X_{ij}=\partial_i
\partial_j X-\Gamma_{ij}^{k}X_k$, where $\Gamma_{ij}^{k}$ is the
Christoffel symbol of the metric on $M^n$. We need the following
identities
\begin{equation}\label{Gauss for}
X_{ij}=-h_{ij}\nu, \quad \quad \mbox{Gauss formula}
\end{equation}

\begin{equation}\label{Wein for}
\nu_{i}=h_{ij}X^j, \quad \quad \mbox{Weingarten formula}
\end{equation}

\begin{equation}\label{Gauss}
R_{ijkl}=h_{ik}h_{jl}-h_{il}h_{jk}, \quad \quad \mbox{Gauss equation}
\end{equation}

\begin{equation}\label{Codazzi}
\nabla_{k}h_{ij}=\nabla_{j}h_{ik}, \quad \quad \mbox{Codazzi
equation.}
\end{equation}
Then, using the Codazzi equation we get
\begin{equation*}
\nabla_{i}\nabla_{j}h_{kl}=\nabla_{i}(\nabla_{j}h_{lk})
=\nabla_{i}(\nabla_{k}h_{lj})=\nabla_{i}\nabla_{k}h_{lj}.
\end{equation*}
Then, using the Ricci identities we have
\begin{equation*}
\nabla_{i}\nabla_{j}h_{kl}=\nabla_{k}\nabla_{i}h_{lj}+R_{iklm}h^{m}_{j}
+R_{ikjm}h^{m}_{l}.
\end{equation*}
Using the Codazzi equation again, it follows that
\begin{eqnarray*}
\nabla_{i}\nabla_{j}h_{kl}&=&\nabla_{k}(\nabla_{l}h_{ji})+R_{iklm}h^{m}_{j}
+R_{ikjm}h^{m}_{l}\\
&=&\nabla_{k}\nabla_{l}h_{ji}+R_{iklm}h^{m}_{j}+R_{ikjm}h^{m}_{l}.
\end{eqnarray*}
Using the Gauss equation, we have
\begin{eqnarray}\label{2rd}
\nabla_{i}\nabla_{j}h_{kl}
&=&\nabla_{k}\nabla_{l}h_{ij}+h^{m}_{j}(h_{il}h_{km}-h_{im}h_{kl})+h^{m}_{l}(h_{ij}h_{km}-h_{im}h_{kj}).
\end{eqnarray}

\subsection{The corresponding scalar equation}

\

In coordinates on the sphere $\mathbb{S}^n$, we equivalently
formulate the problem by the corresponding scalar equation. Since
the initial $C^{2,\gamma}$-hypersurface is star-sharped, there
exists a scalar function $u_0\in C^{2,\gamma} (M^{n})$ such that
 $X_0: M^{n} \rightarrow \mathbb{R}^{n+1}$ has the form $x \mapsto
(x,u_0(x))$. The hypersurface $M_{t}^{n}$ given by the embedding
\begin{eqnarray*}
X(\cdot, t): M^{n}\rightarrow \mathbb{R}^{n+1}
\end{eqnarray*}
at time $t$ may be represented as a graph over $M^n\subset
\mathbb{S}^{n}$, and then we can make ansatz
\begin{eqnarray*}
X(x,t)=\left(x,u(x,t)\right)
\end{eqnarray*}
for some function $u: M^{n} \times [0,T) \rightarrow \mathbb{R}$.

\begin{lemma} \label{lemma2-1}
Define $p:=X(x,t)$ and assume that a point on $\mathbb{S}^{n}$ is
described by local coordinates $\xi^{1},\ldots,\xi^{n}$, that is,
$x=x(\xi^{1},\ldots,\xi^{n})$. Let $\partial_i$ be the corresponding
coordinate fields on $\mathbb{S}^n$ and
$\sigma_{ij}=g_{\mathbb{S}^n}(\partial_i,\partial_j)$ be the metric
on $\mathbb{S}^n$. Let $u_{i}=D_{i}u$, $u_{ij}=D_{j}D_{i}u$, and
$u_{ijk}=D_{k}D_{j}D_{i}u$ denote the covariant derivatives of $u$
with respect to the round metric $g_{\mathbb{S}^n}$ and let $\nabla$
be the Levi-Civita connection of $M_{t}^{n}$ with respect to the
metric $g$ induced from the standard metric of $\mathbb{R}^{n+1}$.
Then, the following formulas hold:

(i) The tangential
vector on $M_{t}$ is
\begin{eqnarray*}
X_{i}=\partial_{i}+u_i\partial_{r}
\end{eqnarray*}
and the corresponding outward unit normal vector is given by
\begin{eqnarray*}
\nu=\frac{1}{v}\left(\partial_r-\frac{1}{u^2}u^j\partial_j\right),
\end{eqnarray*}
where $u^{j}=\sigma^{ij}u_{i}$, and $v:=\sqrt{1+u^{-2}|D u|^2}$
with $D u$ the gradient of $u$.

(ii) The induced metric $g$ on $M_t$ has the form
\begin{equation*}
g_{ij}=u^2\sigma_{ij}+u_{i} u_{j}
\end{equation*}
and its inverse is given by
\begin{equation*}
g^{ij}=\frac{1}{u^2}\left(\sigma^{ij}-\frac{u^i  u^j
}{u^2v^{2}}\right).
\end{equation*}

(iii) The second fundamental form of
$M_t$ is given by
\begin{eqnarray*}
h_{ij}=\frac{1}{v}\left(-u_{ij} +u
\sigma_{ij}+\frac{2}{u}{u_i  u_j }\right).
\end{eqnarray*}
and
\begin{eqnarray*}
h^i_j=g^{ik}h_{jk}=\frac{1}{u v}\delta^i_j-\frac{1}{u
v}\widetilde{\sigma}^{ik}\varphi_{jk},
\qquad\widetilde{\sigma}^{ij}=\sigma^{ij}-\frac{\varphi^i
\varphi^j}{v^2}.
 \end{eqnarray*}
Naturally, the mean curvature is given by
\begin{eqnarray*}
H=\sum_{i=1}^{n}h^i_i=\frac{1}{u
v}\bigg(n-(\sigma^{ij}-\frac{\varphi^i\varphi^j}{v^{2}})\varphi_{ij}\bigg),
\end{eqnarray*}
where $\varphi=\log u$.

(iv) Let $p\in \Sigma^n$, $\hat{\mu}(p)$ be the normal to $\Sigma^n$ at $p$ and $\mu=\mu^i(x) \partial_i$ be the normal to $\partial M^n$ at $x$.  Then
\begin{eqnarray*}
\langle\hat{\mu}(p), \nu(p) \rangle=0 \Leftrightarrow \mu^i(x) u_i(x,t)=0.
\end{eqnarray*}

\end{lemma}

\begin{proof}
Since
\begin{eqnarray*}
h_{ij}=-\langle \overline{\nabla}_{ij}X, \nu\rangle
=-\langle \overline{\nabla}_{\partial_i}\partial_{j}
+u_i\overline{\nabla}_{\partial_j}\partial_{r}+u_j\overline{\nabla}_{\partial_i}\partial_{r}+u_iu_j \overline{\nabla}_{\partial_r}\partial_{r}, \nu\rangle,
\end{eqnarray*}
these formulas can be verified by direct calculation. The details
can also be found in \cite{Ch16}.
\end{proof}

Using techniques as in Ecker \cite{Eck}, (see also \cite{Ge90,
Ge3, Mar}), the problem \eqref{Eq} is degenerated into solving the
following scalar equation with the corresponding initial data
\begin{equation}\label{Eq-}
\left\{
\begin{aligned}
&\frac{\partial u}{\partial t}=\frac{v}{u^{\alpha}H} \qquad
&&~\mathrm{in}~
M^n\times(0,\infty)\\
&\nabla_{\mu} u=0  \qquad&&~\mathrm{on}~ \partial M^n\times(0,\infty)\\
&u(\cdot,0)=u_{0} \qquad &&~\mathrm{in}~M^n.
\end{aligned}
\right.
\end{equation}
By Lemma \ref{lemma2-1}, define a new function $\varphi(x,t)=\log
u(x, t)$ and then the mean curvature can be rewritten as
\begin{eqnarray*}
H=\sum_{i=1}^{n}h^i_i=\frac{e^{-\varphi}}{
v}\bigg(n-(\sigma^{ij}-\frac{\varphi^{i}\varphi^{j}}{v^{2}})\varphi_{ij}\bigg).
\end{eqnarray*}
Hence, the evolution equation in \eqref{Eq-} can be rewritten as
\begin{eqnarray*}
\frac{\partial}{\partial t}\varphi=e^{-\alpha \varphi}
(1+|D\varphi|^2)\frac{1} {[n-(\sigma^{ij}-\frac{\varphi^{i}
\varphi^{j}}{v^2})\varphi_{ij}]}:=Q(\varphi, D\varphi, D^2\varphi).
\end{eqnarray*}
In particular,
 \begin{eqnarray*}
n-(\sigma^{ij}-\frac{\varphi_0^{i}
\varphi_0^{j}}{v^2})\varphi_{0,ij}
 \end{eqnarray*}
 is positive on $M^n$, since $M_0$ is strictly mean convex.
Thus, the problem \eqref{Eq} is again reduced to solve the following
scalar equation with the NBC
\begin{equation}\label{Evo-1}
\left\{
\begin{aligned}
&\frac{\partial \varphi}{\partial t}=Q(\varphi, D\varphi,
D^{2}\varphi) \quad
&& \mathrm{in} ~M^n\times(0,T)\\
&\nabla_{\mu} \varphi =0  \quad && \mathrm{on} ~ \partial M^n\times(0,T)\\
&\varphi(\cdot,0)=\varphi_{0} \quad && \mathrm{in} ~ M^n,
\end{aligned}
\right.
\end{equation}
where
$$n-(\sigma^{ij}-\frac{\varphi_0^{i}
\varphi_0^{j}}{v^2})\varphi_{0,ij}$$ is positive on $M^n$. Clearly,
for the initial surface $M_0$,
$$\frac{\partial Q}{\partial \varphi_{ij}}\Big{|}_{\varphi_0}=\frac{1}{u^{2+\alpha}H^{2}}(\sigma^{ij}-\frac{\varphi_0^{i}
\varphi_0^{j}}{v^2})$$ is positive on $M^n$. Based on the above
facts, as in \cite{Ge90, Ge3, Mar}, we can get the following
short-time existence and uniqueness for the parabolic system
\eqref{Eq}.

\begin{lemma}
Let $X_0(M^n)=M_{0}^{n}$ be as in Theorem \ref{main1.1}. Then there
exist some $T>0$, a unique solution  $u \in
C^{2+\gamma,1+\frac{\gamma}{2}}(M^n\times [0,T])
\cap C^{\infty}(M^n \times (0,T])$, where
$\varphi(x,t)=\log u(x,t)$, to the parabolic system \eqref{Evo-1}
with the matrix
 \begin{eqnarray*}
n-(\sigma^{ij}-\frac{\varphi^{i} \varphi^{j}}{v^2})\varphi_{ij}
 \end{eqnarray*}
positive on $M^n$. Thus there exists a unique map $\psi:
M^n\times[0,T]\rightarrow M^n$ such that $\psi(\partial M^n
,t)=\partial M^n$ and the map $\widehat{X}$ defined by
\begin{eqnarray*}
\widehat{X}: M^n\times[0,T)\rightarrow \mathbb{R}^{n+1}:
(x,t)\mapsto X(\psi(x,t),t)
\end{eqnarray*}
has the same regularity as stated in Theorem \ref{main1.1} and is
the unique solution to the parabolic system \eqref{Eq}.
\end{lemma}

Let $T^{\ast}$ be the maximal time such that there exists some
 \begin{eqnarray*}
u\in C^{2+\gamma,1+\frac{\gamma}{2}}(M^n\times[0,T^{\ast}))\cap
C^{\infty}(M^n\times(0,T^{\ast}))
 \end{eqnarray*}
  which solves \eqref{Evo-1}. In the
sequel, we shall prove a priori estimates for those admissible
solutions on $[0,T]$ where $T<T^{\ast}$.

\section{$C^0$, $\dot{\varphi}$ and gradient estimates}

\begin{lemma}[\bf$C^0$ estimate]\label{lemma3.1}
Let $\varphi$ be a solution of \eqref{Evo-1}. Then for $\alpha>0$,
we have
\begin{equation*}
c_1\leq u(x, t) \Theta^{-1}(t, c) \leq c_2, \qquad\quad \forall~
x\in M^n, \ t\in[0,T],
\end{equation*}
for some positive constants $c_{1}$, $c_{2}$, where $\Theta(t,
c):=\left\{\frac{\alpha t}{n}+e^{\alpha
c}\right\}^{\frac{1}{\alpha}}$ with
 \begin{eqnarray*}
\inf_{M^n}\varphi(\cdot,0)\leq c\leq \sup_{M^n} \varphi(\cdot,0).
\end{eqnarray*}
\end{lemma}

\begin{proof}
Let $\varphi(x, t)=\varphi(t)$ (independent of $x$) be  the solution
of \eqref{Evo-1} with $\varphi(0)=c$. In this case, the first
equation in \eqref{Evo-1} reduces to an ODE
\begin{eqnarray*}
\frac{d}{d t}\varphi=e^{-\alpha\varphi}\frac{1}{n}.
\end{eqnarray*}
Therefore,
\begin{eqnarray}\label{blow}
\varphi(t)=\frac{1}{\alpha}\ln\left(\frac{\alpha t}{n}+e^{\alpha
c}\right), \ \ \mathrm{for}\ \ \alpha>0.
\end{eqnarray}
Using the maximum principle, we can obtain that
\begin{equation}\label{C^0}
\frac{1}{\alpha}\ln\left(\frac{\alpha}{n}t+e^{\alpha\varphi_1}\right)\leq\varphi(x,
t)
\leq\frac{1}{\alpha}\ln\left(\frac{\alpha t}{n}+e^{\alpha\varphi_2}\right),
\end{equation}
where $\varphi_1:=\inf_{M^n}\varphi(\cdot,0)$ and
$\varphi_2:=\sup_{M^n} \varphi(\cdot,0)$. The estimate is obtained
since $\varphi =\log u$.
\end{proof}

\begin{lemma}[\bf$\dot{\varphi}$ estimate]\label{lemma3.2}
Let $\varphi$ be a solution of
\eqref{Evo-1} and $\Sigma^n$ be a smooth, convex cone, then
 for $\alpha>0$,
\begin{eqnarray*}
\min\left\{\inf_{M^n}\dot{\varphi}(\cdot, 0)\cdot\Theta(0)^{\alpha},
\frac{1}{n}\right\} \leq \dot{\varphi}(x, t)\Theta(t)^{\alpha}\leq
\max\left\{\sup_{M^{n}}\dot{\varphi}(\cdot,
0)\cdot\Theta(0)^{\alpha}, \frac{1}{n}\right\}.
\end{eqnarray*}
\end{lemma}

\begin{proof}
Set
\begin{eqnarray*}
\mathcal{M}(x,t)=\dot{\varphi}(x, t)\Theta(t)^{\alpha}.
\end{eqnarray*}
Differentiating both sides of the first evolution equation of
\eqref{Evo-1}, it is easy to get that
\begin{equation} \label{3.4}
\left\{
\begin{aligned}
&\frac{\partial\mathcal{M}}{\partial t}=
Q^{ij}D_{ij}\mathcal{M}+Q^{k}D_k \mathcal{M}+\alpha
\Theta^{-\alpha}\left(\frac{1}{n}-\mathcal{M}\right)\mathcal{M}
\quad
&& \mathrm{in} ~M^n\times(0,T)\\
&\nabla_{\mu}\mathcal{M}=0 \quad && \mathrm{on} ~\partial M^n\times(0,T)\\
&\mathcal{M}(\cdot,0)=\dot{\varphi}_0\cdot\Theta(0)^{\alpha} \quad
&& \mathrm{on} ~ M^n,
\end{aligned}
\right.
\end{equation}
where $Q^{ij}:=\frac{ \partial Q}{\partial \varphi_{ij}}$
 and $Q^k:=\frac{ \partial Q}{\partial \varphi_{k}}$.
Then the result follows from the maximum principle.
\end{proof}

\begin{lemma}[\bf Gradient estimate]\label{Gradient}
Let $\varphi$ be a solution of \eqref{Evo-1} and $\Sigma^n$ be a
smooth, convex cone described as in Theorem \ref{main1.1}. Then we
have for $\alpha>0$,
\begin{equation}\label{Gra-est}
|D\varphi|\leq \sup_{M^n}|D\varphi(\cdot, 0)|, \qquad\quad \forall~
x\in M^n, \ t\in[0,T].
\end{equation}
\end{lemma}

\begin{proof}
Set $\psi=\frac{|D \varphi|^2}{2}$. By differentiating  $\psi$, we
have
\begin{equation*}
\begin{aligned}
\frac{\partial \psi}{\partial t}
=\frac{\partial}{\partial t}\varphi_m \varphi^m
= \dot{\varphi}_m\varphi^m
=Q_m \varphi^m.
\end{aligned}
\end{equation*}
Then using the evolution equation of $\varphi$ in (\ref{Evo-1})
yields
\begin{eqnarray*}
\frac{\partial \psi}{\partial t}=Q^{ij}\varphi_{ijm} \varphi^m
+Q^k\varphi_{km} \varphi^m-\alpha Q|D \varphi|^2.
\end{eqnarray*}
Interchanging the covariant derivatives, we have
\begin{equation*}
\begin{aligned}
\psi_{ij}&=D_j(\varphi_{mi} \varphi^m)\\&=\varphi_{mij} \varphi^m+\varphi_{mi} \varphi^m_j\\
&=(\varphi_{ijm}-R^l_{jmi}\varphi_{l})\varphi^m+\varphi_{mi}\varphi^m_j.
\end{aligned}
\end{equation*}
Therefore, we can express
$\varphi_{ijm} \varphi^m$
as
\begin{eqnarray*}
\varphi_{ijm} \varphi^m
=\psi_{ij}+R^l_{jmi}\varphi_l \varphi^m-\varphi_{mi} \varphi^m_j.
\end{eqnarray*}
Then, in view of the fact
$R_{jmil}=\sigma_{ji}\sigma_{ml}-\sigma_{lj}\sigma_{im}$ on
$\mathbb{S}^n$, we have
\begin{equation}\label{gra}
\begin{aligned}
\frac{\partial \psi}{\partial t}&=Q^{ij}\psi_{ij}+Q^k \psi_k
-Q^{ij}(\sigma_{ij}|D\varphi|^2-\varphi_i \varphi_j)\\&-Q^{ij}\varphi_{mi} \varphi^{m}_{j}-\alpha Q|D \varphi|^2
.\end{aligned}
\end{equation}

Since the matrix $Q^{ij}$ is positive definite, the third and the
fourth terms in the RHS of \eqref{gra} are non-positive. Noticing
that the last term in the RHS of \eqref{gra} is also non-positive if
$\alpha>0$. Since $\Sigma^n$ is convex, using a similar argument to
the proof of \cite[Lemma 5]{Mar} (see pp. 1308) implies that
\begin{eqnarray*}
\nabla_{\mu}\psi=-\sum\limits_{i,j=1}^{n-1}h_{ij}^{\partial
M^{n}}\nabla_{e_i}\varphi\nabla_{e_j}\varphi \leq
0~~~~\qquad\mathrm{on}~\partial M^n\times(0,T),
\end{eqnarray*}
where an orthonormal frame at $x\in\partial M^{n}$, with
$e_{1},\ldots,e_{n-1}\in T_{x}\partial M^{n}$ and $e_{n}:=\mu$, has
been chosen for convenience in the calculation, and
$h_{ij}^{\partial M^{n}}$ is the second fundamental form of the
boundary $\partial M^{n}\subset\Sigma^{n}$.
 So, we can get
\begin{equation*}
\left\{
\begin{aligned}
&\frac{\partial \psi}{\partial t}\leq Q^{ij}\psi_{ij}+Q^k\psi_k
\qquad &&\mathrm{in}~
M^n\times(0,T)\\
&\nabla_{\mu} \psi \leq 0   && \mathrm{on}~\partial M^n\times(0,T)\\
&\psi(\cdot,0)=\frac{|D\varphi(\cdot,0)|^2}{2} \qquad
&&\mathrm{in}~M^n.
\end{aligned}\right .\end{equation*}
Using the maximum principle, we get the gradient estimate of
$\varphi$ in Lemma \ref{Gradient}.
\end{proof}

\begin{remark}
\rm{It is worth pointing out that  the evolving surface $M_{t}^{n}$
is always star-shaped under the assumption of Theorem \ref{main1.1},
since, by Lemma \ref{Gradient}, we have
\begin{eqnarray*}
\left\langle \frac{X}{|X|}, \nu\right\rangle=\frac{1}{v}
\end{eqnarray*}
is bounded from below by a positive constant.}
\end{remark}

Combing the gradient estimate with $\dot{\varphi}$ estimate, we can
obtain
\begin{corollary}
If $\varphi$ satisfies \eqref{Evo-1}, then we have
\begin{eqnarray}\label{w-ij}
0<c_3\leq
H\Theta
\leq c_4<+\infty,
\end{eqnarray}
where $c_3$ and $c_4$ are positive constants independent of
$\varphi$.
\end{corollary}

\section{H\"{o}lder Estimates and Convergence}

\

Set $\Phi=\frac{1}{|X|^{\alpha}H}$, $w=\langle X, \nu\rangle$ and
$\Psi= \frac{\Phi}{w}$. We can get the following evolution
equations.

\begin{lemma}\label{EVQ}
Under the assumptions of Theorem \ref{main1.1}, we have
\begin{eqnarray*}
\frac{\partial}{\partial t}g_{ij}=2\Phi h_{ij},
\end{eqnarray*}
\begin{eqnarray*}
\frac{\partial}{\partial t}g^{ij}=-2\Phi h^{ij},
\end{eqnarray*}
\begin{equation*}
\frac{\partial}{\partial t}\nu=-\nabla \Phi,
\end{equation*}
\begin{equation*}
\begin{split}
\partial_{t}h_{i}^{j}-\Phi H^{-1}\Delta h_{i}^{j}&=\Phi H^{-1}|A|^2 h_{i}^{j}-\frac{2\Phi}{H^2}H_i H^j
-2\Phi h_{ik}h^{kj}\\
&-\alpha\Phi(\nabla_{i}\log u \nabla^{j}\log H+\nabla^{j}\log u \nabla_{i}\log H)\\
&+\alpha\Phi u^{-1} u_{i}^{j}-\alpha (\alpha+1)\Phi\nabla_{i}\log u \nabla^{j}\log u,
\end{split}
\end{equation*}
and
\begin{equation}\label{div-for-1}
\begin{split}
 \frac{\partial \Psi }{\partial t}&=\mbox{div}_g (u^{-\alpha} H^{-2} \nabla \Psi)-2H^{-2} u^{-\alpha} \Psi^{-1} |\nabla \Psi|^2\\
 &-\alpha \Psi^2- \alpha \Psi^2 u^{-1 }\nabla^i u \langle X, X_i \rangle- \alpha u^{-\alpha-1} H^{-2} \nabla_iu \nabla^i\Psi.
 \end{split}
\end{equation}
\end{lemma}

\begin{proof}
The  first three evolution equations are easy to obtain and so are omitted.
Using the Gauss formula, we have
\begin{equation*}
\begin{split}
\partial_{t}h_{ij}&=
\partial_{t}\langle \partial_i \partial_j X, -\nu\rangle\\
&=\langle \partial_i \partial_j (\Phi \nu), -\nu\rangle
-\langle \Gamma_{ij}^{k}\partial_{k}X-h_{ij}\nu, \partial_{t}\nu\rangle\\
&=-\partial_i \partial_j \Phi-\Phi\langle\partial_i \partial_j \nu, \nu\rangle
+\Gamma_{ij}^{k}\Phi_k\\
&=-\nabla^{2}_{ij} \Phi-\Phi\langle\partial_i (h_{j}^{k}\partial_{k}X), \nu\rangle\\
&=-\nabla^{2}_{ij} \Phi+\Phi h_{ik}h_{j}^{k}.
\end{split}
\end{equation*}
Direct calculation results in
\begin{equation*}
\begin{split}
\nabla^{2}_{ij}\Phi&=\Phi(-\frac{1}{H}H_{ij}+\frac{2H_i H_j}{H^2})\\
&+\alpha\Phi(\nabla_{i}\log u \nabla_{j}\log H+\nabla_{j}\log u \nabla_{i}\log H)\\
&-\alpha \Phi u^{-1} u_{ij}+\alpha (\alpha+1)\Phi \nabla_{i}\log u \nabla_{j}\log u.
\end{split}
\end{equation*}
Since
\begin{eqnarray*}
\Delta h_{ij}=H_{ij}+H h_{ik}h^{k}_{j}-h_{ij}|A|^2,
\end{eqnarray*}
so
\begin{equation*}
\begin{split}
\nabla^{2}_{ij}\Phi&=-\Phi H^{-1}\Delta h_{ij}+\Phi h_{ik}h^{k}_{j}
-\Phi H^{-1}|A|^2 h_{ij}+\frac{2H_i H_j}{H^2}\\
&+\alpha\Phi(\nabla_{i}\log u \nabla_{j}\log H+\nabla_{j}\log u \nabla_{i}\log H)\\
&-\alpha \Phi u^{-1} u_{ij}+\alpha (\alpha+1)\Phi\nabla_{i}\log u \nabla_{j}\log u.
\end{split}
\end{equation*}
Thus,
\begin{equation*}
\begin{split}
\partial_{t}h_{ij}-\Phi H^{-1}\Delta h_{ij}&=\Phi H^{-1}|A|^2 h_{ij}-\frac{2\Phi}{H^2}H_i H_j\\
&-\alpha\Phi(\nabla_{i}\log u \nabla_{j}\log H+\nabla_{j}\log u \nabla_{i}\log H)\\
& +\alpha \Phi u^{-1} u_{ij}-\alpha (\alpha+1)\Phi\nabla_{i}\log u \nabla_{j}\log u.
\end{split}
\end{equation*}
Then
\begin{equation*}
\begin{split}
\partial_tH&=  \partial_t g^{ij} h_{ij}+ g^{ij} \partial_t h_{ij}\\
&= -2\Phi h^{ij}h_{ij} +g^{ij} \left(\Phi H^{-1}\Delta h_{ij} +  \Phi H^{-1}|A|^2 h_{ij}-\frac{2\Phi}{H^2} \nabla_ iH \nabla_j H\right)\\
& + \alpha \Phi g^{ij}  \left(-\nabla_{i}\log u \nabla_{j}\log H-\nabla_{j}\log u \nabla_{i}\log H +  u^{-1} u_{ij}
- (\alpha+1)\nabla_{i}\log u \nabla_{j}\log u \right)\\
&=\Phi H^{-1}\Delta H- \frac{2\Phi}{H^2} |\nabla H|^2-\Phi |A|^2+ \alpha \Phi g^{ij}  \left(-2\nabla_{i}\log u \nabla_{j}\log H +  u^{-1} u_{ij}
- (\alpha+1)\nabla_{i}\log u \nabla_{j}\log u \right)\\
&=u^{-\alpha} H^{-2} \Delta H- 2u^{-\alpha} H^{-3} |\nabla H|^2- u^{-\alpha} H^{-1} |A|^2- 2\alpha u^{-\alpha-1} H^{-2} \nabla_iu \nabla^iH+ \alpha u^{-\alpha-1} H^{-1} \Delta u\\
& \quad -\alpha( \alpha+1) u^{-\alpha-2} H^{-1} |\nabla u|^2.
\end{split}
\end{equation*}
Clearly,
\begin{eqnarray*}
\partial_{t}w= \Phi+\alpha \Phi u^{-1} \nabla^i u\langle X,
X_i\rangle  + \Phi  H^{-1} \nabla^i H\langle X,X_i\rangle,
\end{eqnarray*}
using the Weingarten equation, we have
\begin{eqnarray*}
w_i=h_{i}^{k}\langle X, X_k\rangle,
\end{eqnarray*}
\begin{eqnarray*}
w_{ij}=h_{i,\ j}^{k}\langle X,
X_k\rangle+h_{ij}-h_{i}^{k}h_{kj}\langle X, \nu\rangle =h_{ij,
k}\langle X, X^k\rangle+h_{ij}-h_{i}^{k}h_{kj}\langle X, \nu\rangle.
\end{eqnarray*}
Thus,
\begin{eqnarray*}
\Delta w= H+\nabla^i H\langle X, X_i\rangle-|A|^2 \langle X, \nu\rangle
\end{eqnarray*}
 and
\begin{eqnarray*}
\partial_t w=u^{-\alpha} H^{-2} \Delta w+ u^{-\alpha} H^{-2} w |A|^2+ \alpha u^{-\alpha-1} H^{-1} \nabla^i u\langle X, X_i \rangle.
\end{eqnarray*}
Hence
\begin{equation*}
\begin{split}
\frac{\partial \Psi }{\partial t}&= - \alpha \frac{1}{u^{1+\alpha}} \frac{1}{Hw} \dot{u}-  \frac{1}{u^{\alpha} H^2} \frac{1}{w} \partial_tH -\frac{1}{u^{\alpha} H} \frac{1}{w^2} \partial_tw\\
&=- \alpha \frac{1}{u^{1+\alpha}} \frac{1}{Hw} \frac{1}{u^{\alpha-1} H w}-   \frac{1}{u^{\alpha} H^2} \frac{1}{w} \partial_tH -\frac{1}{u^{\alpha} H} \frac{1}{w^2} \partial_tw\\
&=- \alpha u^{-2\alpha} H^{-2} w^{-2} +\alpha( \alpha+1) u^{-2\alpha-2} H^{-3} w^{-1} |\nabla u|^2
+ 2u^{-2\alpha} H^{-5}w^{-1} |\nabla H|^2\\
& \quad + 2\alpha u^{-2\alpha-1} H^{-4}w^{-1} \nabla_iu \nabla^{i}H
- \alpha u^{-2\alpha-1} H^{-3} w^{-1} \Delta u -u^{-2\alpha}
H^{-4}w^{-1} \Delta H - u^{-2\alpha} H^{-3}w^{-2} \Delta w \\
& \quad - \alpha u^{-2\alpha-1} H^{-2} w^{-2} \nabla^i u\langle X,
X_i \rangle.
\end{split}
\end{equation*}
In order to prove (\ref{div-for-1}), we calculate
$$\nabla_i \Psi =-\alpha u^{-\alpha-1} H^{-1} w^{-1} \nabla_i u-u^{-\alpha} H^{-2} w^{-1} \nabla_i H- u^{-\alpha} H^{-1} w^{-2} \nabla_i w$$
and
\begin{equation*}
\begin{split}
\nabla^2_{ij}\Psi&= \alpha(\alpha+1) u^{-\alpha-2} H^{-1} w^{-1} \nabla_i u \nabla_j u +\alpha u^{-\alpha-1} H^{-2} w^{-1} \nabla_i u \nabla_j H+ \alpha u^{-\alpha-1} H^{-1} w^{-2} \nabla_i u \nabla_j w\\
& \quad -\alpha  u^{-\alpha-1} H^{-1} w^{-1} \nabla^2_{ij} u+ \alpha u^{-\alpha-1} H^{-2} w^{-1} \nabla_i H \nabla_j u+ 2u^{-\alpha} H^{-3} w^{-1} \nabla_i H \nabla_j H\\
& \quad + u^{-\alpha} H^{-2} w^{-2} \nabla_i H \nabla_j w- u^{-\alpha} H^{-2} w^{-1}\nabla^2_{ij} H + \alpha u^{-\alpha-1} H^{-1} w^{-2} \nabla_i w \nabla_j u\\
& \quad + u^{-\alpha} H^{-2} w^{-2} \nabla_i w \nabla_j H+
2u^{-\alpha} H^{-1} w^{-3} \nabla_i w \nabla_j w - u^{-\alpha}
H^{-1} w^{-2} \nabla^2_{ij} w.
\end{split}
\end{equation*}
Thus
\begin{equation*}
\begin{split}
u^{-\alpha} H^{-2} \Delta \Psi
&=\alpha(\alpha+1) u^{-2\alpha-2} H^{-3} w^{-1} |\nabla u|^2+ 2u^{-2\alpha} H^{-5} w^{-1} |\nabla H|^2+2u^{-2\alpha} H^{-3} w^{-3} |\nabla w|^2\\
&+2\alpha u^{-2\alpha-1} H^{-4} w^{-1} \nabla_i u \nabla^i H+2 \alpha u^{-2\alpha-1} H^{-3} w^{-2} \nabla_i u \nabla^i w+2u^{-2\alpha} H^{-4} w^{-2} \nabla_i H \nabla^i w  \\
&-\alpha  u^{-2\alpha-1} H^{-3} w^{-1} \Delta u-u^{-2\alpha} H^{-4} w^{-1}\Delta H-u^{-2\alpha} H^{-3} w^{-2} \Delta w.
\end{split}
\end{equation*}
So we have
\begin{equation*}
\begin{split}
&\mbox{div} (u^{-\alpha} H^{-2} \nabla \Psi)= -\alpha u^{-\alpha-1} H^{-2} \nabla_i \Psi \nabla^i u- 2 u^{-\alpha} H^{-3} \nabla_i \Psi \nabla^i H+ u^{-\alpha} H^{-2} \Delta \Psi\\
&=(2\alpha^2+\alpha) u^{-2\alpha-2} H^{-3} w^{-1}|\nabla u|^2+5\alpha u^{-2\alpha -1} H^{-4} w^{-1} \nabla_i u \nabla^i H+ 3\alpha u^{-2\alpha -1} H^{-3} w^{-2} \nabla_i u \nabla^i w\\
& +4 u^{-2\alpha} H^{-5} w^{-1} |\nabla H|^2+ 4 u^{-2\alpha} H^{-4} w^{-2} \nabla_i w \nabla^i H +2u^{-2\alpha} H^{-3} w^{-3} |\nabla w|^2\\
&-\alpha  u^{-2\alpha-1} H^{-3} w^{-1} \Delta u-u^{-2\alpha} H^{-4} w^{-1}\Delta H-u^{-2\alpha} H^{-3} w^{-2} \Delta w
\end{split}
\end{equation*}
and
\begin{equation*}
\begin{split}
2H^{-1} w |\nabla \Psi|^2&=2 \alpha^2 u^{-2\alpha-2} H^{-3} w^{-1} |\nabla u|^2+ 2u^{-2\alpha} H^{-5} w^{-1} |\nabla H|^2 + 2u^{-2\alpha} H^{-3} w^{-3} |\nabla w|^2 \\
&+4 \alpha u^{-2\alpha-1} H^{-4} w^{-1} \nabla_i u \nabla^i H+4\alpha u^{-2\alpha-1} H^{-3} w^{-2} \nabla_i u  \nabla^i w+4 u^{-2\alpha} H^{-4} w^{-2} \nabla_i H \nabla^i w.
\end{split}
\end{equation*}
As above, we have
\begin{equation*}
\begin{split}
&\frac{\partial \Psi }{\partial t}-\mbox{div} (u^{-\alpha} H^{-2} \nabla \Psi)+2H^{-1} w |\nabla \Psi|^2\\
&=- \alpha u^{-2\alpha} H^{-2} w^{-2}-\alpha u^{-2\alpha-1} H^{-2} w^{-2} \nabla^i u\langle X, X_i \rangle+ \alpha^2 u^{-2\alpha-2} H^{-3} w^{-1} |\nabla u|^2\\
& \qquad +\alpha u^{-2\alpha-1} H^{-4}w^{-1} \nabla_iu \nabla^iH+\alpha u^{-2\alpha-1} H^{-3} w^{-2} \nabla_i u  \nabla^i w\\
&=-\alpha \Psi^2- \alpha \Psi^2 u^{-1 }\nabla^i u\langle X, X_i \rangle- \alpha u^{-\alpha-1} H^{-2} \nabla_iu \nabla^i\Psi.
\end{split}
\end{equation*}
The proof is finished.
\end{proof}

Now, we define the  rescaled flow by
\begin{equation*}
\widetilde{X}=X\Theta^{-1}.
\end{equation*}
Thus,
\begin{equation*}
\widetilde{u}=u\Theta^{-1},
\end{equation*}
\begin{equation*}
\widetilde{\varphi}=\varphi-\log\Theta,
\end{equation*}
and the rescaled Gauss curvature is given by
\begin{equation*}
\widetilde{H}=H\Theta.
\end{equation*}
Then, the rescaled scalar curvature equation takes the form
\begin{equation*}
\frac{\partial}{\partial t}\widetilde{u}=\frac{v}{\widetilde{u}^{\alpha}\widetilde{H}}\Theta^{-\alpha}
-\frac{1}{n}\widetilde{u}\Theta^{-\alpha}.
\end{equation*}
Defining $t=t(s)$ by the relation
\begin{equation*}
\frac{dt}{d s}=\Theta^{\alpha}
\end{equation*}
such that $t(0)=0$ and $t(S)=T$. Then $\widetilde{u}$ satisfies
\begin{equation}\label{Eq-re}
\left\{
\begin{aligned}
&\frac{\partial}{\partial
s}\widetilde{u}=\frac{v}{\widetilde{u}^{\alpha}\widetilde{H}}-\frac{\widetilde{u}}{n}
\qquad && \mathrm{in}~
M^n\times(0,S)\\
&\nabla_{\mu} \widetilde{u}=0  \qquad && \mathrm{on}~ \partial M^n\times(0,S)\\
&\widetilde{u}(\cdot,0)=\widetilde{u}_{0}  \qquad &&
\mathrm{in}~M^n.
\end{aligned}
\right.
\end{equation}

\begin{lemma}\label{res-01}
Let $X$ be a solution of (\ref{Eq}) and $\widetilde{X}=X \Theta^{-1}$ be the rescaled solution. Then
\begin{equation*}
\begin{split}
&D \widetilde{u}=D u \Theta^{-1}, ~~~~D \widetilde{\varphi}=D \varphi,~~~~ \frac{\partial \widetilde{u}}{\partial s}=\frac{ \partial u}{\partial t} \Theta^{\alpha-1}-\frac{1}{n} u\Theta^{-1},\\
&\widetilde{g}_{ij}= \Theta^{-2}g_{ij},~~~~\widetilde{g}^{ij}=\Theta^2 g^{ij},~~~~\widetilde{h}_{ij}=h_{ij}\Theta^{-1}.
\end{split}
\end{equation*}
\end{lemma}
\begin{proof}
These relations can be computed directly.
\end{proof}

\begin{lemma} \label{lemma4-3}
Let $u$ be a solution to the parabolic system \eqref{Evo-1}, where
$\varphi(x,t)=\log u(x,t)$, and $\Sigma^n$ be a smooth, convex cone
described as in Theorem \ref{main1.1}. Then there exist some
$\beta>0$ and some $C>0$ such that the rescaled function
$\widetilde{u}(x,s):=u(x,t(s)) \Theta^{-1}(t(s))$ satisfies
\begin{equation}
 [D \widetilde{u}]_{\beta}+\left[\frac{\partial \widetilde{u}}{\partial s}\right]_{\beta}+[\widetilde{H}]_{\beta}\leq C(\parallel u_{0}\parallel_{C^{2+\gamma,1+\frac{\gamma}{2}}(M^n)}, n, \beta, M^n),
\end{equation}
where $[f]_{\beta}:=[f]_{x,\beta}+[f]_{s,\frac{\beta}{2}}$ is the
sum of the H\"{o}lder coefficients of $f$ in $M^n\times[0,S]$ with
respect to $x$ and $s$.
\end{lemma}

\begin{proof}
We divide our proof in  three steps\footnote{In the proof of Lemma
\ref{lemma4-3}, the constant $C$ may differ from each other.
However, we abuse the symbol $C$ for the purpose of convenience.}.

\textbf{Step 1:} We need to prove that
\begin{equation*}
  [D \widetilde{u}]_{x,\beta}+[D \widetilde{u}]_{s,\frac{\beta}{2}}\leq C(\parallel u_0\parallel_{ C^{2+\gamma,1+\frac{\gamma}{2}}(M^n)}, n, \beta, M^n).
\end{equation*}
According to Lemmas \ref{lemma3.1}, \ref{lemma3.2} and
\ref{Gradient}, it follows that
$$|D \widetilde{u}|+\left|\frac{\partial \widetilde{u}}{\partial s}\right|\leq C(\parallel u_0\parallel_{ C^{2+\gamma,1+\frac{\gamma}{2}}(M^n)}, M^n).$$
Then we can easily obtain the bound of $[\widetilde{u}]_{\beta}$ for
any $0<\beta<1$. Lemma 3.1 in \cite[Chap. 2]{La2} implies that the
bound for $[D\widetilde{u}]_{s,\frac{\beta}{2}}$ follows from a
bound for $[\widetilde{u}]_{s,\frac{\beta}{2}}$ and
$[D\widetilde{u}]_{x,\beta}$. Hence it remains to bound  $[D
\varphi]_{x,\beta}$ since $D\widetilde{u}=\widetilde{u}
D\widetilde{\varphi}$. For this,  fix $s$ and the equation
(\ref{Evo-1}) can be rewritten as an elliptic Neumann problem
 \begin{equation} \label{key1}
   \mbox{div}_{\sigma}\left(\frac{D \widetilde{\varphi}}{\sqrt{1+|D\widetilde{\varphi}|^2}}\right)=\frac{n}{\sqrt{1+|D\widetilde{\varphi}|^2}}- e^{-\alpha \widetilde{\varphi}} \frac{\sqrt{1+|D\widetilde{\varphi}|^2}}{\dot{\widetilde{\varphi}}+ \frac{1}{n}}.
 \end{equation}
In fact, the equation (\ref{key1}) is of the form
$D_i(a^i(p))+a(x)=0$, where the bound of $a$, the smallest and
largest eigenvalues  of $a^{ij}(p):=\frac{\partial a^i}{\partial
p^j}$ are controlled due to the estimate for $|D\varphi|$ and
$|\widetilde{u}|$. The estimate of $[D
\widetilde{\varphi}]_{x,\beta}$ for some $\beta$ follows from a
Morrey estimate by calculations similar to the arguments in
\cite[Chap. 4, $\S$ 6; Chap. 10, $\S$ 2 ]{La2} (interior estimate
and boundary estimate). For a rigorous proof of this estimate the
reader is referred to \cite{Mar1}.

\textbf{Step 2:} The next thing to do is to show that
\begin{equation*}
  \left[\frac{\partial \widetilde{u}}{\partial s}\right]_{x,\beta}+\left[\frac{\partial \widetilde{u}}{\partial s}\right]_{s,\frac{\beta}{2}}\leq
  C(\parallel u_0\parallel_{ C^{2+\gamma,1+\frac{\gamma}{2}}(M^n)}, n, \beta, M^n),
\end{equation*}
  As $\frac{\partial}{\partial s}\widetilde{u}=\widetilde{u}\left(\frac{v}{\widetilde{u}^{1+\alpha} \widetilde{H}}-\frac{1}{n}\right)$, it is enough to bound
  $\left[\frac{v}{\widetilde{u}^{1+\alpha} \widetilde{H}}\right]_{\beta}$.
Set $w(s):=\frac{v}{\widetilde{u}^{1+\alpha} \widetilde{H}}=
\Theta^{\alpha}\Psi$, and then we have
$$\frac{\partial w}{\partial s} = \frac{\partial}{\partial t}( \Theta^{\alpha}\Psi) \frac{\partial t}{\partial s}= \frac{\alpha}{n}w+ \Theta^{2\alpha} \frac{\partial \Psi}{\partial t}.$$
Let $\widetilde{\nabla}$ be the Levi-Civita connection of
$\widetilde{M}_{s}:=\widetilde{X}(M^n,s)$ with respect to the metric
$\widetilde{g}$. Combine with (\ref{div-for-1}) and Lemma
\ref{res-01}, we get
\begin{equation}\label{div-form-02}
\begin{split}
\frac{\partial w}{\partial s} &=\mbox{div}_{\widetilde{g}} (\widetilde{u}^{-\alpha} \widetilde{H}^{-2} \widetilde{\nabla} w)-2 \widetilde{H}^{-2} \widetilde{u}^{-\alpha} w^{-1} |\widetilde{\nabla} w|^2_{\widetilde{g}}\\
&+\frac{\alpha}{n}w-\alpha w^2- \alpha  w^2 P- \alpha  \widetilde{u}^{-\alpha-1} \widetilde{H}^{-2} \widetilde{\nabla}_i\widetilde{u} \widetilde{\nabla}^iw,
\end{split}
\end{equation}
where $P:=u^{-1}\nabla^i u\langle X, X_i \rangle$. Applying Lemma
\ref{Gradient}, we have
$$|P|\leq |\nabla u|_g=\frac{| D\varphi|}{v}\leq C.$$
The weak formulation of (\ref{div-form-02}) is
\begin{equation}\label{div-form-03}
\begin{split}
\int_{s_0}^{s_1} \int_{\widetilde{M}_s}  \frac{\partial w }{\partial s}  \eta d\mu_s ds &
=\int_{s_0}^{s_1} \int_{\widetilde{M}_s} \mbox{div}_{\widetilde{g}} (\widetilde{u}^{-\alpha} \widetilde{H}^{-2} \widetilde{\nabla} w) \eta
-2 \widetilde{H}^{-2} \widetilde{u}^{-\alpha} w^{-1} |\widetilde{\nabla} w|^2_{\widetilde{g}} \eta d\mu_s ds\\
&+\int_{s_0}^{s_1} \int_{\widetilde{M}_s} (\frac{\alpha}{n}w-\alpha w^2- \alpha  w^2 P
-\alpha  \widetilde{u}^{-\alpha-1} \widetilde{H}^{-2} \widetilde{\nabla}_i\widetilde{u} \widetilde{\nabla}^iw) \eta d\mu_s ds.
\end{split}
\end{equation}
Since $\nabla_{\mu} \widetilde{\varphi}=0$, the interior and
boundary estimates are basically the same. We define the test
function $\eta:=\xi^2 w$, where $\xi$ is a smooth function with
values in $[0,1]$ and is supported in a small parabolic
neighborhood. Then
\begin{equation}\label{imcf-hec-for-02}
\begin{split}
\int_{s_0}^{s_1} \int_{\widetilde{M}_s}  \frac{\partial w }{\partial s}  \xi^2 w d\mu_s ds=
\frac{1}{2}\parallel w \xi\parallel_{2,\widetilde{M}_s}^2 \mid_{s_0}^{s_1}-\int_{s_0}^{s^1} \int_{\widetilde{M}_s}  \xi \dot{\xi} w^2 d\mu_s ds.
\end{split}
\end{equation}
Using integration by parts and Young's inequality, we can obtain
\begin{equation}\label{imcf-hec-for-03}
\begin{split}
&\int_{s_0}^{s_1} \int_{\widetilde{M}_s}  \mbox{div}_{\widetilde{g}} (\widetilde{u}^{-\alpha} \widetilde{H}^{-2} \widetilde{\nabla} w)  \xi^2 w  d\mu_sds\\
&=-\int_{s_0}^{s_1} \int_{\widetilde{M}_s}   \widetilde{u}^{-\alpha} \widetilde{H}^{-2} \xi^2\widetilde{\nabla}_i w  \widetilde{\nabla}^iw  d\mu_sds
-2\int_{s_0}^{s_1} \int_{\widetilde{M}_s}\widetilde{u}^{-\alpha} \widetilde{H}^{-2} \xi w\widetilde{\nabla}_iw \widetilde{\nabla}^i \xi d\mu_sds\\
&\leq\int_{s_0}^{s_1} \int_{\widetilde{M}_s}  \widetilde{u}^{-\alpha} \widetilde{H}^{-2} |\widetilde{\nabla} \xi|^2w^2  d\mu_sds
\end{split}
\end{equation}
and
 \begin{equation}\label{imcf-hec-for-04}
\begin{split}
&\int_{s_0}^{s_1} \int_{\widetilde{M}_s}(\frac{\alpha}{n}w-\alpha w^2- \alpha  w^2 P- \alpha  \widetilde{u}^{-\alpha-1} \widetilde{H}^{-2} \widetilde{\nabla}_i\widetilde{u} \widetilde{\nabla}^iw)  \xi^2 w  d\mu_sds\\
& \leq C\alpha \int_{s_0}^{s_1} \int_{\widetilde{M}_s} \xi^2 (w^2+|w|^3)  d\mu_sds+ \int_{s_0}^{s_1} \int_{\widetilde{M}_s} \alpha  \widetilde{u}^{-\alpha-1} \widetilde{H}^{-2} |\widetilde{\nabla}\widetilde{u}| |\widetilde{\nabla}w|  \xi^2 w  d\mu_sds\\
&\leq  C\alpha \int_{s_0}^{s_1} \int_{\widetilde{M}_s} \xi^2 (w^2+|w|^3)  d\mu_sds +
\frac{\alpha}{2} \int_{s_0}^{s_1} \int_{\widetilde{M}_s}  \widetilde{u}^{-\alpha} \widetilde{H}^{-2}  |\widetilde{\nabla}w|^2  \xi^2  d\mu_sds \\
&+ \frac{\alpha}{2} \int_{s_0}^{s_1} \int_{\widetilde{M}_s}   \widetilde{u}^{-\alpha-2} \widetilde{H}^{-2} |\widetilde{\nabla}\widetilde{u}|^2 \xi^2 w^2  d\mu_sds.
\end{split}
\end{equation}
Combing (\ref{imcf-hec-for-02}), (\ref{imcf-hec-for-03}) and
(\ref{imcf-hec-for-04}), we have
 \begin{equation*}
\begin{split}
&\frac{1}{2}\parallel w \xi\parallel_{2,\widetilde{M}_s}^2 \mid_{s_0}^{s_1}
+(2-\frac{\alpha}{2})\int_{s_0}^{s_1} \int_{\widetilde{M}_s} \widetilde{u}^{-\alpha}\widetilde{H}^{-2} |\widetilde{\nabla} w|^2  \xi^2   d\mu_sds\\
& \leq \int_{s_0}^{s_1} \int_{\widetilde{M}_s}  \xi |\dot{\xi}| w^2 d\mu_s ds
+\int_{s_0}^{s_1} \int_{\widetilde{M}_s}  \widetilde{u}^{-\alpha} \widetilde{H}^{-2} |\widetilde{\nabla} \xi|^2w^2  d\mu_sds\\
&+  C\alpha \int_{s_0}^{s_1} \int_{\widetilde{M}_s} \xi^2 ( w^2+|w|^3)  d\mu_sds
+ \frac{\alpha}{2} \int_{s_0}^{s_1} \int_{\widetilde{M}_s}   \widetilde{u}^{-\alpha-2} \widetilde{H}^{-2} |\widetilde{\nabla}\widetilde{u}|^2 \xi^2 w^2  d\mu_sds,
\end{split}
\end{equation*}
 which implies that
 \begin{equation}\label{imcf-hec-for-06}
\begin{split}
&\frac{1}{2}\parallel w \xi\parallel_{2,\widetilde{M}_s}^2 \mid_{s_0}^{s_1}
+\frac{(2-\frac{\alpha}{2})}{\max(\widetilde{u}^{\alpha}\widetilde{H}^{2}) }\int_{s_0}^{s_1} \int_{\widetilde{M}_s} |\widetilde{\nabla} w|^2  \xi^2   d\mu_sds\\
& \leq (1+ \frac{1}{\min(\widetilde{u}^{\alpha} \widetilde{H}^{2})}) \int_{s_0}^{s_1} \int_{\widetilde{M}_s}  w^2 (\xi |\dot{\xi}| +|\widetilde{\nabla} \xi|^2)d\mu_s ds\\
&  +  \alpha \left(C+ \frac{\max(|
\widetilde{\nabla}\widetilde{u}|)^2}{2\min(\widetilde{u}^{2+\alpha}
\widetilde{H}^{2}) }\right) \int_{s_0}^{s_1} \int_{\widetilde{M}_s}
\xi^2 w^2 +\xi^2 |w|^3 d\mu_sds.
\end{split}
\end{equation}
This means that $w$ belong to the De Giorgi class of functions in
$M^n \times [0,S)$. Similar to the arguments in \cite[Chap. 5, \S 1
and \S 7]{La2},  there exist  constants $\beta$ and $C$ such that
$$[w]_{\beta}\leq C \parallel w\parallel_{L^{\infty}(M^n \times [0,S))}\leq  C(\parallel u_0\parallel_{ C^{2+\gamma,1+\frac{\gamma}{2}}(M^n)}, n, \beta, M^n).$$

\textbf{Step 3:} Finally, we have to show that
\begin{equation*}
  [ \widetilde{H}]_{x,\beta}+[\widetilde{H}]_{s,\frac{\beta}{2}}\leq C(\parallel u_0\parallel_{ C^{2+\gamma,1+\frac{\gamma}{2}}(M^n)}, n, \beta, M^n).
\end{equation*}
This follows from the fact that
$$\widetilde{H}=\frac{\sqrt{1+|D\varphi|^2}}{\widetilde{u}^{1+\alpha} w}$$
together with the estimates for $\widetilde{u}$, $w$, $D\varphi$.
\end{proof}

Then we can obtain the following higher-order estimates
\begin{lemma}
Let $u$ be a solution to the parabolic system \eqref{Evo-1}, where
$\varphi(x,t)=\log u(x,t)$, and $\Sigma^n$ be a smooth, convex cone
described as in Theorem \ref{main1.1}. Then for any $s_0\in (0,S)$
there exist some $\beta>0$ and some $C>0$ such that
\begin{equation}\label{imfcone-holder-01}
\parallel\widetilde{u}\parallel_{C^{2+\beta,1+\frac{\beta}{2}}(M^n\times [0,S])}\leq C(\parallel u_0\parallel_{ C^{2+\gamma, 1+\frac{\gamma}{2}}(M^n)}, n, \beta, M^n)
\end{equation}
and for all $k\in \mathbb{N}$,
\begin{equation}\label{imfcone-holder-02}
\parallel\widetilde{u}\parallel_{C^{2k+\beta,k+\frac{\beta}{2}}(M^n\times [s_0,S])}\leq C(\parallel u_0(\cdot, s_0)\parallel_{C^{2k+\beta,k+\frac{\beta}{2}}(M^n)}, n, \beta, M^n).
\end{equation}
\end{lemma}

\begin{proof}
By Lemma \ref{lemma2-1}, we have
$$uvH=n-(\sigma^{ij}-\frac{\varphi^{i}\varphi^{j}}{v^{2}})\varphi_{ij}=n-u^2 \Delta_g \varphi.$$
Since
$$u^2 \Delta_g \varphi=\widetilde{u}^2 \Delta_{\widetilde{g}} \varphi=
-| \widetilde{\nabla} \widetilde{u}|^2+ \widetilde{u} \Delta_{\widetilde{g}} \widetilde{u},$$
then
\begin{equation*}
\begin{split}
\frac{\partial \widetilde{u}}{\partial s}&=\frac{ \partial u}{\partial t} \Theta^{\alpha-1}-\frac{1}{n} \widetilde{u}\\
&=-\frac{uvH}{u^{1+\alpha}H^2} \Theta^{\alpha-1} + \frac{2v}{u^{\alpha}H} \Theta^{\alpha-1}-\frac{1}{n} \widetilde{u}\\
&=\frac{\Delta_{\widetilde{g}} \widetilde{u}}{\widetilde{u}^{\alpha}\widetilde{H}^2}+ \frac{2v}{\widetilde{u}^{\alpha}\widetilde{H}}
-\frac{1}{n} \widetilde{u}
-\frac{n+| \widetilde{\nabla} \widetilde{u}|^2}{\widetilde{u}^{1+\alpha}\widetilde{H}^2},
\end{split}
\end{equation*}
which is  a uniformly parabolic equation with H\"{o}lder continuous
coefficients. Therefore, the linear theory (see \cite[Chap.
4]{Lieb}) yield the inequality (\ref{imfcone-holder-01}).

Set $\widetilde{\varphi}=\log \widetilde{u}$, and then the rescaled
version of the evolution equation in (\ref{Eq-re}) takes the form
\begin{equation*}
  \frac{\partial \widetilde{\varphi}}{\partial s}=e^{-\alpha \widetilde{\varphi}} \frac{ v^2}{ \left[n-\left(\sigma^{ij}-\frac{\widetilde{\varphi}^i\widetilde{\varphi}^j}{v^2}\right) \widetilde{\varphi}_{ij}\right]}-\frac{1}{n},
\end{equation*}
where $v=\sqrt{1+|D \widetilde{\varphi}|^2}$. According to the
$C^{2+\beta,1+\frac{\beta}{2}}$-estimate of $\widetilde{u}$ (see
Lemma \ref{lemma4-3}), we can treat the equations for $\frac{
\partial\widetilde{\varphi}}{\partial s}$ and $D_i
\widetilde{\varphi}$ as second-order linear uniformly parabolic PDEs
on $M^n\times [s_0,S]$. At the initial time $s_0$, all compatibility
conditions are satisfied and the initial function $u(\cdot,t_0)$ is
smooth. We can obtain a $C^{3+\beta, \frac{3+\beta}{2}}$-estimate
for $D_i \widetilde{\varphi}$ and a $C^{2+\beta,
\frac{2+\beta}{2}}$-estimate for $\frac{
\partial\widetilde{\varphi}}{\partial s}$ (the estimates are
independent of $T$) by Theorem 4.3 and Exercise 4.5 in \cite[Chap.
4]{Lieb}. Higher regularity can be proven by induction over $k$.
\end{proof}

\begin{theorem} \label{key-2}
Under the hypothesis of Theorem \ref{main1.1}, we conclude
\begin{equation*}
T^{*}=+\infty.
\end{equation*}
\end{theorem}
\begin{proof}
The proof of this result is quite similar to the corresponding
argument in \cite[Lemma 8]{Mar} and so is omitted.
\end{proof}

\section{Convergence of the rescaled flow}

\ We know that after the long-time existence of the flow has been
obtained (see Theorem \ref{key-2}), the rescaled version of the
system (\ref{Evo-1}) satisfies
\begin{equation}
\left\{
\begin{aligned}
&\frac{\partial}{\partial
s}\widetilde{\varphi}=\widetilde{Q}(\widetilde{\varphi},
D\widetilde{\varphi}, D^2\widetilde{\varphi})  \qquad &&\mathrm{in}~
M^n\times(0,\infty)\\
&\nabla_{\mu} \widetilde{\varphi}=0  \qquad &&\mathrm{on}~ \partial M^n\times(0,\infty)\\
&\widetilde{\varphi}(\cdot,0)=\widetilde{\varphi}_{0} \qquad
&&\mathrm{in}~M^n,
\end{aligned}
\right.
\end{equation}
where
$$\widetilde{Q}(\widetilde{\varphi}, D\widetilde{\varphi},
D^2\widetilde{\varphi}):=e^{-\alpha \widetilde{\varphi}} \frac{
v^2}{
\left[n-\left(\sigma^{ij}-\frac{\widetilde{\varphi}^i\widetilde{\varphi}^j}{v^2}\right)
\widetilde{\varphi}_{ij}\right]}-\frac{1}{n}$$ and
$\widetilde{\varphi}=\log \widetilde{u}$. Similar to what has been
done in the $C^1$ estimate (see Lemma \ref{Gradient}), we can deduce
a decay estimate of $\widetilde{u}(\cdot, s)$ as follows.

\begin{lemma} \label{lemma5-1}
Let $u$ be a solution of \eqref{Eq-}, then for $\alpha>0$, we have
\begin{equation}\label{Gra-est1}
|D\widetilde{u}(x, s)|\leq\sup_{M^n}|D\widetilde{u}(\cdot,
0)|e^{-\lambda s},
\end{equation}
where $\lambda$ is a positive constant.
\end{lemma}

\begin{proof}
Set $\psi=\frac{|D \widetilde{\varphi}|^2}{2}$. Similar to that
in Lemma \ref{Gradient}, we can obtain
\begin{equation}\label{gra-}
\begin{aligned}
\frac{\partial \psi}{\partial
s}=&\widetilde{Q}^{ij} \psi_{ij}+\widetilde{Q}^k \psi_k
-\widetilde{Q}^{ij}(\sigma_{ij}|D\widetilde{\varphi}|^2-\widetilde{\varphi}_i
\widetilde{\varphi}_j)\\&-\widetilde{Q}^{ij}\widetilde{\varphi}_{mi}
\widetilde{\varphi}^{m}_{j}-\alpha\widetilde{Q}|D
\widetilde{\varphi}|^2,
\end{aligned}
\end{equation}
 with the boundary condition
\begin{equation*}
\begin{aligned}
D_ \mu \psi\leq 0.
\end{aligned}
\end{equation*}
By the $C^2$ estimate, we can find a positive constant $\lambda$
such that
\begin{equation*}
\left\{
\begin{aligned}
&\frac{\partial \psi}{\partial s}\leq
\widetilde{Q}^{ij}\psi_{ij}+\widetilde{Q}^k\psi_k-\lambda \psi
\quad &&\mathrm{in}~
M^n\times(0,\infty)\\
&D_\mu \psi \leq 0 \quad &&\mathrm{on}~\partial M^n\times(0,\infty)\\
&\psi(\cdot,0)=\frac{|D\widetilde{\varphi}(\cdot,0)|^2}{2}
\quad&&\mathrm{in}~M^n.
\end{aligned}\right.
\end{equation*}
Using the maximum principle and Hopf's lemma, we can get the
gradient estimates of $\widetilde{\varphi}$, and then the inequality
(\ref{Gra-est1}) holds from the estimate for $D\widetilde{u}$.
\end{proof}

\begin{lemma}\label{rescaled flow}
Let $u$ be a solution of the flow
\eqref{Eq-}. Then,
\begin{equation*}
\widetilde{u}(\cdot, s)
\end{equation*}
converges to a real number as $s\rightarrow +\infty$.
\end{lemma}
\begin{proof}
Set $f(t):= \mathcal{H}^n(M_{t}^{n})$, which, as before, represents
the $n$-dimensional Hausdorff measure of $M_{t}^{n}$ and is actually
the area of $M_{t}^{n}$. According to the first variation of a
submanifold, see e.g. \cite{ls},
 and the fact $\mbox{div}_{M_{t}^{n}}\nu=H$, we have
\begin{equation}\label{imcf-crf-for-01}
\begin{split}
f'(t)&=\int_{M_{t}^{n}} \mbox{div}_{M_{t}^{n}} \left( \frac{\nu}{|X|^{\alpha} H}\right) d\mathcal{H}^n\\
&=\int_{M_{t}^{n}} \sum_{i=1}^n \left\langle \nabla_{e_i}\left(\frac{\nu}{|X|^{\alpha} H}\right), e_i\right\rangle d\mathcal{H}^n\\
&=\int_{M_{t}^{n}}  |u|^{-\alpha} d\mathcal{H}^{n},
\end{split}
\end{equation}
where $\{e_{i}\}_{1\leq i\leq n}$ is some orthonormal basis of the
tangent bundle $TM_{t}^{n}$. We know that (\ref{C^0}) implies
$$\left(\frac{\alpha}{n}t+ e^{\alpha \varphi_2}\right)^{-1}\leq u^{-\alpha}\leq \left(\frac{\alpha}{n}t+ e^{\alpha \varphi_1}\right)^{-1},$$
where $\varphi_1=\inf_{M^n} \varphi(\cdot,0)$ and
$\varphi_2=\sup_{M^n} \varphi(\cdot,0)$. Hence
$$\left(\frac{\alpha}{n}t+ e^{\alpha \varphi_2}\right)^{-1} f(t)\leq f'(t) \leq \left(\frac{\alpha}{n}t+ e^{\alpha \varphi_1}\right)^{-1}f(t).$$
Combining this fact with (\ref{imcf-crf-for-01}) yields
 \begin{eqnarray*}
\frac{(\frac{\alpha}{n}t+ e^{\alpha \varphi_2})^{\frac{n}{\alpha}}
\mathcal{H}^n(M_{0}^{n})}{e^{n \varphi_2}} \leq f(t)\leq
\frac{(\frac{\alpha}{n}t+ e^{\alpha \varphi_1})^{\frac{n}{\alpha}}
\mathcal{H}^n(M_{0}^{n})}{ e^{n \varphi_1}}.
 \end{eqnarray*}
 Therefore, the rescaled
hypersurface $\widetilde{M}_s=M_{t}^{n} \Theta^{-1}$ satisfies the
following inequality
 \begin{eqnarray*}
\frac{(\frac{\alpha}{n}t+ e^{\alpha \varphi_2})^{\frac{n}{\alpha}}
\mathcal{H}^n(M_{0}^{n})}{(\frac{\alpha}{n}t+ e^{\alpha
c})^{\frac{n}{\alpha}} e^{n \varphi_2}} \leq
\mathcal{H}^n(\widetilde{M}_s)\leq \frac{(\frac{\alpha}{n}t+
e^{\alpha \varphi_1})^{\frac{n}{\alpha}}
\mathcal{H}^n(M_{0}^{n})}{(\frac{\alpha}{n}t+ e^{\alpha
c})^{\frac{n}{\alpha}} e^{n \varphi_1}},
 \end{eqnarray*}
 which implies that the area
of  $\widetilde{M}_s$ is bounded and the bounds are independent of
$s$. Together with Lemma \ref{imfcone-holder-01}, Lemma
\ref{lemma5-1} and the Arzel\`{a}-Ascoli theorem, we conclude that
$\widetilde{u}(\cdot,s)$ must converge in $C^{\infty}(M^n)$ to a
constant function $r_{\infty}$ with
\begin{eqnarray*}
\frac{1}{e^{\varphi_{2}}}\left(\frac{\mathcal{H}^n(M_{0}^{n})}{\mathcal{H}^n(M^{n})}\right)^{\frac{1}{n}}\leq
r_{\infty}
\leq\frac{1}{e^{\varphi_{1}}}\left(\frac{\mathcal{H}^n(M_{0}^{n})}{\mathcal{H}^n(M^{n})}\right)^{\frac{1}{n}},
\end{eqnarray*}
i.e.,
\begin{eqnarray}\label{radius}
\frac{1}{\sup\limits_{M^{n}}u_{0}}\left(\frac{\mathcal{H}^n(M_{0}^{n})}{\mathcal{H}^n(M^{n})}\right)^{\frac{1}{n}}\leq
r_{\infty}
\leq\frac{1}{\inf\limits_{M^{n}}u_{0}}\left(\frac{\mathcal{H}^n(M_{0}^{n})}{\mathcal{H}^n(M^{n})}\right)^{\frac{1}{n}}.
\end{eqnarray}
This completes the proof.
\end{proof}

So, we have
\begin{theorem}\label{rescaled flow}
The rescaled flow
\begin{equation*}
\frac{d
\widetilde{X}}{ds}=\frac{1}{|\widetilde{X}|^{\alpha}\widetilde{H}}\nu-\frac{\widetilde{X}}{n}
\end{equation*}
exists for all time and the leaves converge in $C^{\infty}$ to a
piece of round sphere of radius $r_{\infty}$, where $r_{\infty}$
satisfies (\ref{radius}).
\end{theorem}

\vspace{0.2 cm}

$\\$\textbf{Acknowledgements}. This research was supported in part
by the National Natural Science Foundation of China (Grant Nos.
11801496 and 11926352), China Scholarship Council, the Fok Ying-Tung
Education Foundation (China), and Hubei Key Laboratory of Applied
Mathematics (Hubei University). The first author, Prof. J. Mao,
wants to thank the Department of Mathematics, IST, University of
Lisbon for its hospitality during his visit from September 2018 to
September 2019. The authors would like to thank Prof. Li Chen and
Dr. Di Wu for useful discussions during the preparation of this
paper.

\end{document}